\documentclass[letterpaper, 10 pt, conference]{ieeeconf}  

\IEEEoverridecommandlockouts                              
\usepackage{graphics} 
\usepackage{epsfig} 
\usepackage{amsmath} 
\usepackage{amssymb}  
\usepackage[dvipsnames]{xcolor}

\usepackage{stmaryrd} 
\usepackage{bbm} 
\usepackage{algorithmic}
\usepackage[linesnumbered,ruled,vlined]{algorithm2e}
\usepackage{caption}
\usepackage{subcaption}
\usepackage{booktabs}
\usepackage{dsfont}
\usepackage{mathtools}
\usepackage[normalem]{ulem}

\newcommand{\bsalpha}{\boldsymbol{\alpha}}

\newcommand{\RR}{\mathbb{R}}
\newcommand{\EE}{\mathbb{E}}

\usepackage{blkarray}

\newcommand\cL{\mathcal L}

\newcommand\bK{\boldsymbol{K}}
\newcommand\bI{\boldsymbol{I}}
\newcommand\bS{\boldsymbol{S}}

\newcommand\bu{\boldsymbol{u}}
\newcommand\bp{\boldsymbol{p}}

\newcommand\bZ{\boldsymbol{Z}}

\newcommand{\nqq}{q_{e^{\prime}, e}^x\left(\hat{\phi}^{n, x}\left(t, e^{\prime}\right), \hat{Z}_{t, \bK}^{n, x}, \hat{Z}_{t, \bI}^{n, x}\right)}
\newcommand{\qq}{q_{e^{\prime}, e}^x\left(\hat{\phi}^{x}\left(t, e^{\prime}\right), \hat{Z}_{t, \bK}^{x}, \hat{Z}_{t, \bI}^{x}\right)}
\newcommand{\npp}{\hat{p}^{n, x}(t,e')}
\newcommand{\pp}{\hat{p}^{x}(t,e')}
\newcommand{\uu}{\hat{u}^{x}(s, \cdot)}
\newcommand{\nuu}{\hat{u}^{n, x}(s, \cdot)}
\newcommand{\bddlambdaI}{\bar{\lambda}_{\bI}}
\newcommand{\bddlambdaK}{\bar{\lambda}_{\bK}}

\newcommand\bbA{\mathbb A}

\newtheorem{theorem}{Theorem}[section]

\newtheorem{assumption}[theorem]{Assumption}
\newtheorem{definition}[theorem]{Definition}
\newtheorem{lemma}[theorem]{Lemma}

\makeatletter
\let\NAT@parse\undefined
\makeatother
\usepackage[hyphens]{url}
\usepackage[colorlinks]{hyperref}
\usepackage[hyphenbreaks]{breakurl}

\title{\LARGE \bf
Modeling of Rumor Propagation in Large Populations with Network via Graphon Games
}

\author{Huaning Liu and G\"ok{\c c}e Dayan{\i}kl{\i}
\thanks{Department of Statistics,
  University of Illinois at Urbana-Champaign, 
  Champaign, IL 61820, USA 
        {\tt\small huaning3@illinois.edu}}
\thanks{Department of Statistics,
  University of Illinois at Urbana-Champaign, 
  Champaign, IL 61820, USA 
        {\tt\small gokced@illinois.edu}}
}

\begin{document}

\maketitle
\thispagestyle{empty}
\pagestyle{empty}

\begin{abstract}
In this paper, we propose a graphon game model to understand how rumor (such as fake news) propagates in large populations that are interacting on a network and how different policies affect the spread. We extend the SKIR model that is used to model rumor propagation and implement individual controls and weighted interactions with other agents to have controlled dynamics. The agents aim to minimize their own expected costs non-cooperatively. We give the finite player game model and the limiting graphon game model to approximate the Nash equilibrium in the population. We give the graphon game Nash equilibrium as a solution to a continuum of ordinary differential equations (ODEs) and give existence results. Finally, we give a numerical approach and analyze examples where we use piecewise constant graphon.

\end{abstract}

\section{Introduction}\label{sec:introduction}
Rumor propagation is a long-standing research topic that has evolved significantly over the years with the increased use of social media and its contribution to the rapid propagation of fake news. Early models such as the Daley-Kendall (DK) model and the Maki-Thompson (MT) model laid the groundwork by describing rumor spread using probabilistic approaches \cite{daley-kendall, Maki1973MathematicalMA}. Various subsequent finite-state rumor propagation models were inspired by these two fundamental ones, which capture the complexities of rumor dynamics from different perspectives (e.g. \cite{6785909,Bettencourt_2006,XIA2015295}). The field has further branched out into optimal control problems, where the goal is to devise strategies to control the spread of rumors effectively, mostly from the perspective of a regulator or the government \cite{optc,DONG2022112711}.

Psychologists have identified several psychological factors underlying rumor propagation, such as the increased uncertainty, lack of control, anxiety, which can make individuals spread fake news to increase their perceived control \cite{dif_psych}. While calming their negative feelings by propagating news, people are taking risks of acquiring bad reputation. Under the trade-off between personal benefits of sharing rumors and the risk of reputational damage, this behavior can be modeled as a non-cooperative game, where each individual minimizes their own personal costs. In large populations, the challenge lies in tracking interactions among agents, complicating the prediction of aggregate behavior and equilibrium identification. In order to overcome these challenges, mean field game (MFG~\cite{MFGG,huang2006large}) offers a robust method for modeling large-scale interactions in non-cooperative games \cite{MFInfection}. However, it simplifies the analysis by assuming that individuals are indistinguishable from one another and interact with each other symmetrically. These assumptions are not realistic in the context of rumor propagation, since the social network that rumors spread by is inherently asymmetric and the individuals have different personality traits.

We will use \textit{graphon game} to address this shortcoming, motivated by the heterogeneity it offers ~\cite{DBLP:journals/corr/abs-1802-00080, carmona2019stochasticgraphongamesi,caines_graphon,carmona_lq_graphon,bayraktar_graphon}. This framework has been effectively used for finite state models, with application in epidemics \cite{graphon_epidemics}. In the real world, a social network can be viewed as a graphon, where each node represents an individual, and the connections between nodes represent the social interaction level between them. That is, the dynamics of the state of each node (i.e., the individual) are influenced by both their neighbors and their own control. We are interested in understanding how the dynamics of the entire network evolve when a small portion of individuals are \textit{infected} with a rumor at the beginning. Our goal is to analyze the overall dynamics of the group and determine how Nash equilibria are formed under these conditions. Our aim is then to design \textit{policies} that can minimize the spread of \textit{rumors} (such as \textit{fake news}). This approach allows us to capture the complexity of social interactions and the strategic behavior of individuals in response to rumor propagation within a large-scale network. 

Our contributions are three fold. First, inspired by the system dynamics models of opinion dynamics and rumor propagation, we introduce a graphon game model to incorporate the \textit{control} of the individuals and their \textit{interactions} on a network that is modeled by a graphon. Second, we give the (continuum of) forward-backward ordinary differential equation (FBODE) system that deduces the Nash equilibrium in the graphon game and give the existence results. Finally, we present an example with a special graphon (i.e., piecewise constant graphon) and give a numerical algorithm to find the Nash equilibrium to analyze the effects of different policies.

The rest of the paper is organized as follows: In Section~\ref{sec:model}, we introduce the finite-player model of rumor propagation network game and the corresponding graphon game. In Section~\ref{sec:theory}, we provide the main theoretical results i.e., the characterization of Nash equilibrium with FBODEs and the existence result. In Section~\ref{sec:algo-num}, we describe the computational method and provide numerical results for the piecewise constant graphon example.
Finally, in Section~\ref{sec:conclusion-future} we give the concluding remarks and our future work.
\begin{figure}[t]
    \centering
\includegraphics[width=0.5\linewidth]{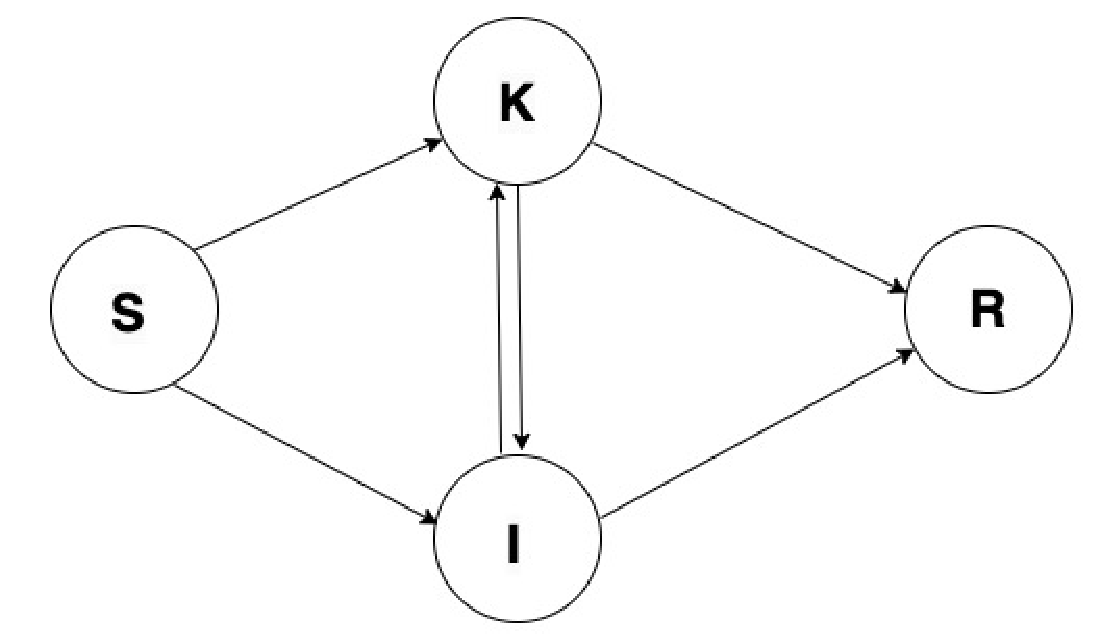}
    \caption{Diagram of the SKIR Model Transitions}
    \label{fig:1}
\end{figure}

\section{Model}
\label{sec:model}

\par{\textbf{States. }}The epidemiological Susceptible-Infected-Removed (SIR) model has been proved a good fit for rumor propagation by representing people with three states: uninformed (S), actively spreading (I), and uninterested (R) \cite{daley-kendall}. We use one of its extension, the Susceptible-Known-Infected-Removed (SKIR) model \cite{xiao}. It introduces anti-rumor information to better capture the dynamics of both rumor and anti-rumor in social networks, where the extra \textit{known} (\textbf{K}) state is actively spreading the truth, and the \textit{infected} (\textbf{I}) state is actively spreading the fake news. 
\par{\textbf{Transitions. }} Figure~\ref{fig:1} gives a diagram of state transitions. First, an uninformed (\textbf{S}) player transits to K or I upon meeting a \textbf{K} or \textbf{I} player. Second, a \textbf{K} or \textbf{I} player can switch to the other state by meeting each other. We denote $\beta$ as the \textit{base} meeting intensity of players. It could be player-specific, state-specific and/or time-dependent i.e. a mapping from players' index set, state set, and/or $[0,T]$ to $\mathbb{R}_+$. In this paper, we will take it as state dependent: $\beta^{\textbf{S}}, \beta^{\textbf{K}}, \beta^{\textbf{I}}$. Third, the transition from state \textbf{K} (resp. \textbf{I}) to \textbf{R}, namely forgetting, happens by an exponentially distributed time with constant rate $\mu_\textbf{K} \in (0,\bar{\mu}]$ (resp. $\mu_\textbf{I}$) for some $\bar{\mu} > 0$.

\subsection{Finite Player Model}
\label{subsec:model_finite}
We first introduce the model with finitely many players, then generalize to the limit case. Let $T > 0$ be the finite time horizon. We denote the set of $A$-valued admissible strategies by $\mathbb{A}$ where $A\subset \RR$\footnote{We assume controls are Markovian, square integrable, and $\RR_+$ valued functions.}, the state of agent $j \in\{1, \ldots, N\}=:\llbracket N \rrbracket$ at time $t$ by $X_t^{j, N}$, and her control by $\alpha_{t}^{j}\in A$. Let the interaction strength between player $i$ and $j$, $i,j \in \llbracket N \rrbracket$ to be $w_{ij}$, where $w$ is a $N$-node undirected dense graph. Then, for example the transition rate for agent $j$ from state \textbf{K} to \textbf{I} is
\small
\begin{equation}
\label{eq:trans_K_I}
  \beta_{\textbf{K}}^{j} \alpha_{t}^{j} Z_{t, \textbf{I}}^{j, N} := \beta_{\textbf{K}}^{j} \alpha_{t}^{j} \frac{1}{N} \sum_{i=1}^N w_{i j} \cdot \alpha_{t}^{i} \cdot \mathbbm{1}_{\textbf{I}}\big(X_{t}^{i,N}\big). 
\end{equation}
\normalsize
This means that agent $j$ transitions from state \textbf{K} to \textbf{I} depending on the base meeting rate of individuals in state \textbf{K}, her own control (i.e., communication rate) $\alpha_t^j$, and the \textit{weighted} average communication rates of individuals in state \textbf{I}. For simplicity in notation, we defined aggregate $Z_{t,\textbf{I}}^{j, N}$ in~\eqref{eq:trans_K_I} and similarly we define $Z_{t,\textbf{K}}^{j, N}:=\frac{1}{N} \sum_{i=1}^N w_{i j} \cdot \alpha_{t}^{i} \cdot \mathbbm{1}_{\textbf{K}}(X_{t}^{i,N})$. We emphasize that the \textit{natural} communication rate is equal to 1 and individuals can choose higher or lower than natural communication rates.\footnote{We emphasize that when we have symmetric interactions (i.e., $w_{ij}=1$ for all $i,j\in \llbracket N \rrbracket$) and when individuals choose the natural communication rate ($\alpha_t^j=1$ for all $t\in[0,T], j \in\llbracket N \rrbracket$), the SKIR model is recovered.} The transition rate matrix for player $j$ is written as follows
\small
\begin{equation*}
Q(\alpha_{t}^{j}, Z_{t, \textbf{K}}^{j, N}, Z_{t, \textbf{I}}^{j, N}) = 
\begin{blockarray}{cccccc}
& \textbf{S} & \textbf{K} & \textbf{I} & \textbf{R} \\
\begin{block}{c(ccccc)}
  \textbf{S} & \cdots & \beta_{\textbf{S}}^{j} \alpha_{t}^{j} Z_{t, \textbf{K}}^{j, N} & \beta_{\textbf{S}}^{j} \alpha_{t}^{j} Z_{t, \textbf{I}}^{j, N} & 0 \\
  \textbf{K} & 0 & \cdots & \beta_{\textbf{K}}^{j} \alpha_{t}^{j} Z_{t, \textbf{I}}^{j, N} & \mu_{\textbf{K}} \\
  \textbf{I} & 0 & \beta_{\textbf{I}}^{j} \alpha_{t}^{j} Z_{t, \textbf{K}}^{j, N} & \cdots & \mu_{\textbf{I}} \\
  \textbf{R} & 0 & 0 & 0 & \cdots \\
\end{block}
\end{blockarray}
\end{equation*}
\normalsize
where notation $\cdots$ represents the negative of the sum of the elements on the same row to satisfy the condition of having the row sum equal to 0. If the rate from state $e_1$ to another state $e_2$, for all $e_1, e_2 \in\{\textbf{S},\textbf{K},\textbf{I}, \textbf{R}\}$ is equal to $q_{12}$, it means that the agent transitions from state $e_1$ to state $e_2$ after an exponentially distributed time with rate $q_{12}$. Here, the heterogeneity of the agents is reflected via the sequence of aggregate variables $\{(Z_{t, \textbf{K}}^{j, N}, Z_{t, \textbf{I}}^{j, N})\}_{j=1,\ldots,N}$.

We then introduce the cost function of the individuals. Consider for player $j \in \llbracket N \rrbracket$, the following running and terminal costs denoted respectively by $f^{j}:[0, T] \times E \times \mathbb{R} \times A \rightarrow \mathbb{R}$ and by $g^{j}: E \times \mathbb{R} \rightarrow \mathbb{R}$, where $E := \{\textbf{S}, \textbf{K}, \textbf{I}, \textbf{R}\}$:
\small
\begin{equation*}
\begin{aligned}
    f^{j}(t, e, z, \alpha)&= \frac{1}{2}(\lambda_{\textbf{I}}(t)-\alpha)^2  \mathbbm{1}_{\{e=\textbf{I}\}}-\lambda_{\textbf{K}}(t)  \mathbbm{1}_{\{e=\textbf{K}\}} +\frac{1}{2}(1-\alpha)^2\\
    g^{j}(e, z)&=c_j  \mathbbm{1}_{\{e = \textbf{I}\}}.
\end{aligned}
\end{equation*}
\normalsize

The motivation for the running cost is as follows: Firstly, despite some researchers suggesting that it might be best from a governmental perspective to maintain a certain ratio of people who know true news and fake news rather than completely eradicating fake news \cite{tani}, various legislations and court cases across countries indicate that the excessive spread of rumors is subject to punishment, For instance, in the United States, it may result in fines, while in China, it could lead to administrative detention~\cite{chinesenews, uslaw}. Therefore, considering the control as the communication rate, when a player is actively spreading fake news, there is a \textit{legal risk} and we penalize the distance between their communication rate and some threshold communication rate denoted by $\lambda_{\textbf{I}}: [0, T] \rightarrow \mathbb{R}$. It could be determined by the government or some regulator for rumor mitigation. Secondly, evidence suggests that avoiding fake news benefits mental health \cite{Alonzo2021-yb}, so when a player knows and spreads true news, a reward is provided instead of a penalty, denoted by $\lambda_{\textbf{K}}: [0, T] \rightarrow \mathbb{R}_{+}$. In addition, recalling that the natural communication rate is $1$, we thus penalize the distance between player's current communication rate and $1$, representing the cost of putting higher or lower levels of effort in communication. The natural communication rate can depend on the individuals' introversion/extroversion levels and therefore can be designed as a function of player index; however, for the sake of simplicity in notation we take it as $1$.
Individuals may also face terminal penalty for ending up spreading fake news; therefore, we penalize agents who are at state \textbf{I} at $T$ with a constant agent-specific $c_j \geq 0$. 

We denote with $\underline{\bsalpha}^N := \{(\bsalpha^1, \ldots, \bsalpha^N)\}$ the control profile of $N$ players. We use $\bsalpha^j=(\alpha_t^j)_{t\in[0,T]}$ and $\bsalpha^{-j}$ to denote agent $j$'s control and control profile of every other agent, respectively. 
Then the objective of agent $j \in \llbracket N \rrbracket$ is to minimize the following expected cost over her control $\bsalpha^j$:

\small
\begin{equation*}
\begin{aligned}
\mathcal{J}^{j, N}(\bsalpha^j; \bsalpha^{-j})=\mathbb{E}\Big[\int_0^T &f^{j}\left(t, X_t^{j, N}, \left(Z_{t, \mathbf{K}}^{j, N}, Z_{t, \mathbf{I}}^{j, N}\right), \alpha^{j}_{t}\right) d t\\
&+g^{j}\left(X_T^{j, N},\left(Z_{T, \mathbf{K}}^{j, N}, Z_{T, \mathbf{I}}^{j, N}\right)\right)\Big].   
\end{aligned}
\end{equation*}
\normalsize

\begin{definition}
    The control profile $\underline{\boldsymbol{\alpha}}^N$ is an $N$-player Nash equilibrium if it is admissible and no player can gain from a unilateral deviation, i.e.,
\small
$$
\mathcal{J}^{j, N}\left(\bsalpha^j ; \bsalpha^{-j}\right) \leq \mathcal{J}^{i, N}\left(\sigma; \bsalpha^{-i}\right), \ \forall i \in \llbracket N\rrbracket \ \forall \sigma \in \mathbb{A}.
$$
\normalsize
\end{definition}

\subsection{Graphon Game Model}
\label{subsec:model_graphon}

When the number of players is very large ($N \rightarrow \infty$), we introduce the graphon game model to approximate the Nash equilibrium. The agent set is now a continuum of non-identical agents, namely $I:=[0,1]$. For some agent $x \in I$, the $E$-valued jump process $\left(X_t^{\bsalpha,x}\right)_{t \in[0, T]}$ denotes her state trajectory, which is potentially influenced by the whole strategy profile ${\bsalpha} := \left(\boldsymbol{\alpha}^x\right)_{x \in I}$ via players' interactions on graphon.  Formally, a graphon is a symmetric Borel-measurable function, $w: I \times I \rightarrow[0,1]$, where $w(x,y)$ represents the connection strength between agents $x$ and $y$. Intuitively, graphon represents the limit of a dense graph when the number of nodes goes to infinity.

Following the graphon game literature, a general aggregate for agent $x \in I$, induced by the graphon, can be written as follows:
\small
\begin{equation}
    Z_t^{{\bsalpha}, x}=\int_I w(x, y) \mathbb{E}\left[K\left(\alpha_t^y, X_{t}^{{\bsalpha}, y}\right)\right] d y,
\end{equation}
\normalsize
where $K(\alpha,x)$ is an interaction function.
Under the case of rumors propagation introduced in \ref{subsec:model_finite}, the aggregate variables for agent $x\in I$ in the limit can be specified as 

\small
\begin{equation}\left\{\begin{array}{l}Z_{t, \textbf{I}}^{{\bsalpha}, x}=\int_I w(x, y) \int_{A} \alpha_t^y \rho_t^y(da, \textbf{I}) d y, \\[1mm] Z_{t, \textbf{K}}^{{\bsalpha}, x}=\int_I w(x, y) \int_{A} \alpha_t^y \rho_t^y(da, \textbf{K}) d y.\end{array}\right.\label{eq:(2)}\end{equation}
\normalsize
These aggregates intuitively give the weighted communication rate of individuals in state \textbf{I} and \textbf{K}, respectively.

The Q-matrix form stays functionally the same, where the finite player aggregates are replaced with the limit aggregates given in equation~\eqref{eq:(2)}. We denote the agents' initial state distribution by $p_{0}^{x}$, the dynamics of the state can be written in a compact form as follows
\small
\begin{equation}
    \frac{d}{d t} p^{{\boldsymbol{\alpha}}, x}(t)=p^{{\boldsymbol{\alpha}}, x}(t) Q^x\left(\alpha_t^x, Z_{t, \textbf{K} }^{{\boldsymbol{\alpha}}, x}, Z_{t, \textbf{I}}^{{\boldsymbol{\alpha}} ,x}\right),
\end{equation}
\normalsize
where $p^{{\boldsymbol{\alpha}}, x}(t):=\left(p^{{\boldsymbol{\alpha}}, x}(t, e)\right)_{e \in E}$ with initial condition $p^{{\boldsymbol{\alpha}}, x}(0)=p_0^x$. The entry of Q-matrix from state $e$ to $e^{\prime}$ for agent $x$ at time $t$ when the agent uses control $\alpha$ will be denoted as $q_t^x(e, e^{\prime}, Z_{t, \textbf{K} }^{{\boldsymbol{\alpha}}, x}, Z_{t, \textbf{I}}^{{\boldsymbol{\alpha}} ,x}, \alpha)$.

The expected cost for agent $x \in I$ with control $\boldsymbol{\sigma} \in 
\mathbb{A}$ while the population chooses $\bsalpha$ is
\small
\begin{equation*}
\begin{aligned}
    \mathcal{J}^x(\boldsymbol{\sigma} ; (\boldsymbol{Z}^{\boldsymbol{\alpha}, x})_{\textbf{K}, \textbf{I}})=\mathbb{E}\Big[\int_0^T &f^x\left(t, X_t^{{\boldsymbol{\alpha}}, x}, Z_{t,\textbf{K}}^{{\boldsymbol{\alpha}}, x}, Z_{t,\textbf{I}}^{{\boldsymbol{\alpha}}, x}, \sigma_t\right) d t\\
    &+g^x\left(X_T^{{\boldsymbol{\alpha}}, x}, Z_{T, \textbf{K}}^{{\boldsymbol{\alpha}}, x}, Z_{T, \textbf{I}}^{{\boldsymbol{\alpha}}, x}\right)\Big].
    \end{aligned}
\end{equation*}
\normalsize
For simplicity, we define $(\bZ^{\bsalpha,x})_{\textbf{K}, \textbf{I}} := (Z_{t, \textbf{K}}^{\bsalpha, x}, Z_{t, \textbf{I}}^{\bsalpha, x})_{t\in[0,T]}$.
\begin{definition}
    The strategy profile ${\bsalpha}$ with $\bsalpha^x\in \bbA$ is a graphon game Nash equilibrium if no player can gain from a unilateral deviation, i.e.,
    \small
\begin{equation*}
\mathcal{J}^x\left(\boldsymbol{\alpha}^x ; (\boldsymbol{Z}^{\boldsymbol{\alpha}, x})_{\textbf{K}, \textbf{I}}\right) \leq \mathcal{J}^x\left(\boldsymbol{\sigma} ; (\boldsymbol{Z}^{\boldsymbol{\alpha}, x})_{\textbf{K}, \textbf{I}}\right), \quad \forall x \in I, \forall \boldsymbol{\sigma} \in \mathbb{A}.    
\end{equation*}
\normalsize
\end{definition}

\section{Main Theoretical Results}
\label{sec:theory}

In this section, we characterize the graphon game Nash equilibrium as a solution of a continuum of FBODE system. Furthermore, we give the existence results. We start by stating our assumptions.

\begin{assumption}
\label{assu:fbode_nash}
\begin{itemize}
    \item[(i.)] We assume $\lambda_{\mathbf{I}}(t)$ and $\lambda_{\mathbf{K}}(t)$ are continuous functions of $t$. We also assume $\alpha_t^j \in[0, \bar{A}]$ for all $t \in[0, T]$ and $j \in \llbracket N \rrbracket$ where $\bar{A}$ denotes an upper bound on the communication rate.
    \item[(ii.)] We assume further $\lambda_{\mathbf{I}}(t)$ 
 is a lipschitz continuous function of $t$ with lipschitz constant $L_{\lambda_{\bI}}$, and $|\lambda_{\textbf{K}}(t)|$ and $|\lambda_{\textbf{I}}(t)|$ are uniformly bounded with bounds $\bar{\lambda}_{\bK} > 0$ and  $\bar{\lambda}_{\bI} > 0$, respectively. Furthermore, $c^x$ is bounded by $\bar{c} > 0$ for all $x \in I$. Finally, meeting intensities $\beta_{\bS}$, $\beta_{\bK}$ and $\beta_{\bI}$ are bounded by some $\bar{\beta} > 0$.
\end{itemize}
\end{assumption}

\begin{theorem}
\label{the:fbode_nash}

Under Assumption~\ref{assu:fbode_nash}.(i), the graphon game Nash equilibrium control profile is written as $\hat{\alpha}^x_t = \hat{\alpha}^x(t,e,(z)_{\textbf{K},\textbf{I}}, u^x(t,\cdot))=:\hat{\phi}^x(t,e)$ for all $x\in I, t\in[0,T]$ where 
\small
\begin{equation}
    \begin{aligned}
    \hat{\phi}^x(t,\textbf{S})=&1+\beta_{\textbf{S}} Z_{t, \textbf{K}}^x\left(u^x(t, \textbf{S})-u^x(t, \textbf{K})\right)\\
    &\hskip1.7mm+\beta_{\textbf{S}} Z_{t,\textbf{I}}^x\left(u^x(t, \textbf{S})-u^x(t, \textbf{I})\right), \\
    \hat{\phi}^x(t, \textbf{K})=&1+\beta_{\textbf{K}} Z_{t, \textbf{I}}^x\left(u^x(t, \textbf{K})-u^x(t, \textbf{I})\right), \\
    \hat{\phi}^x(t,\textbf{I})=&\frac{1}{2}\left(\lambda_{\textbf{I}}(t)+1+\beta_{\textbf{I}} Z^{x}_{t,\textbf{K}}\left(u^x(t, \textbf{I})-u^x(t, \textbf{K})\right)\right), \\
    \hat{\phi}^x(t,\textbf{R}) =& 1,
    \end{aligned}
\end{equation}
\normalsize
if the couple $(\bu,\bp)$ solves the following FBODE system:
\small
\begin{equation*}
    \begin{aligned}
        \dot{p}^x(t, e)=&\sum_{e^{\prime}\in E} p^x(t, \cdot) q^x_t(e^{\prime}, e, Z_{t, \mathbf{K}}^{\boldsymbol{\alpha}, x}, Z_{t, \mathbf{I}}^{\boldsymbol{\alpha}, x}, \hat{\phi}^x(t,\cdot)), e \in\{\textbf{S}, \textbf{K}, \textbf{I}, \textbf{R}\} \\
    \dot{u}^x(t, \textbf{S})=&\beta_{\textbf{S}} \hat{\phi}^x(t, \textbf{S}) Z^{\boldsymbol{\alpha},x}_{t, \textbf{K}}\left(u^x(t, \textbf{S})-u^x(t, \textbf{K})\right) -\frac{1}{2}\big(1-\hat{\phi}^x(t, \textbf{S})\big)^2 \\
        &+\beta_{\textbf{S}} \hat{\phi}^x(t, \textbf{S}) Z^{\boldsymbol{\alpha},x}_{t, \textbf{I}}\left(u^x(t, \textbf{S})-u^x(t, \textbf{I})\right)\\
    \dot{u}^x(t, \textbf{K})=&\beta_{\textbf{K}} \widehat{\phi}^x(t, \textbf{K}) Z_{t, \mathbf{I}}^{\boldsymbol{\alpha}, x} \left(u^x(t, \textbf{K})-u^x(t, \textbf{I})\right) \\
    &+ \mu_{\textbf{K}}\left(u^x(t, \textbf{K})-u^x(t, \textbf{R})\right) + \lambda_{\textbf{K}}(t)-\frac{1}{2}\big(1-\hat{\phi}^{x}(t, \textbf{K})\big)^2 \\
 \dot{u}^x(t, \textbf{I})=&\beta_{\textbf{I}} \hat{\phi}^x(t, \textbf{I}) Z^{\boldsymbol{\alpha},x}_{t, \textbf{K}} \left(u^x(t, \textbf{I})-u^x(t, \textbf{K})\right) \\
 &+ \mu_{\textbf{I}}\left(u^x(t, \textbf{I})-u^x(t, \textbf{R})\right) -\frac{1}{2}\left(1-\hat{\phi}^x(t, \textbf{I})\right)^2 \\&- \frac{1}{2}\left(\lambda_{\textbf{I}}(t)-\hat{\phi}^{\hat{x}}(t, \textbf{I})\right)^2 \\
    \dot{u}^x(t, \textbf{R})=&0 \\
    Z_{t, \textbf{I}}^{{\bsalpha}, x}=&\int_I w(x, y) \hat{\phi}^y(t,\textbf{I}) p^y(t, \textbf{I}) d y\\
        Z_{t, \textbf{K}}^{{\bsalpha}, x}=&\int_I w(x, y) \hat{\phi}^y(t,\textbf{K}) p^y(t, \textbf{K}) d y\\
     u^x(T, e)=&0,\ \forall e \in\{\textbf{S}, \textbf{K}, \textbf{R}\}, u^x(T, \textbf{I})=c^x,\\  
     p^x(0, e)=&p_0^x(e),\ e \in\{\textbf{S}, \textbf{K}, \textbf{I}, \textbf{R}\}, \quad\forall x\in I.
    \end{aligned}
\end{equation*}
\normalsize

\end{theorem}

\begin{proof}
Extending results given in~\cite[Section 7.2]{carmona2018probabilistic} and following similarly to~\cite{graphon_epidemics}, we can write the finite state versions of the Hamilton-Jacobi-Bellman (HJB) and Kolmogorov-Fokker-Planck (KFP) equations. We first start by finding the writing the Hamiltonian for each agent $x$ for all $x\in I$:
\small
\begin{equation*}
\begin{aligned}
    &H^x(t, e, (z)_{\textbf{K},\textbf{I}}, u, \alpha) \\ &\hskip5mm=\sum_{e^\prime\in E} q_t(e, e^{\prime}, (z)_{\textbf{K},\textbf{I}}, \alpha)u(e^{\prime}) + f^x(t, e, (z)_{\textbf{K},\textbf{I}}, \alpha)
\end{aligned}
\end{equation*}
\normalsize
where $q_t(e, e^{\prime}, (z)_{\textbf{K},\textbf{I}}, \alpha)$ is the element of Q-matrix that gives the transition rate from state $e$ to $e^\prime$. The equilibrium control is denoted by $\hat{\phi}^x(t,e):=\hat{\alpha}^x(t,e,(z)_{\textbf{K},\textbf{I}}, u^x(t,\cdot))$ is found by minimizing the Hamiltonian and plugging in $u^x(t,\cdot)$ instead of $u(\cdot)$, for all $t\in[0,T], x\in I$. Then, the HJB can be written as
\small
\begin{equation*}
    \dot{u}^x(t, e) = - H(t, e, (z)_{\textbf{K},\textbf{I}}, u^x(t, \cdot), \hat{\phi}^x(t,e)), \ t\in[0,T], e\in E.
\end{equation*}
\normalsize
Normally, HJB is a partial differential equation; however, since the state space is finite, it appears as an ODE system. Here the function $u^x(t,e)$ is the value function of agent $x$:
\small
\begin{equation*}
\begin{aligned}
    u^x(t,e)= \inf_{\bsalpha^x \in \bbA} \EE \Big[\int_t^T &f^x\left(s, X_s^{{\boldsymbol{\alpha}}, x}, Z_{s,\textbf{K}}^{{\boldsymbol{\alpha}}, x}, Z_{s,\textbf{I}}^{{\boldsymbol{\alpha}}, x}, \alpha^x_s\right) d s\\
    &+g^x\left(X_T^{{\boldsymbol{\alpha}}, x}, Z_{T, \textbf{K}}^{{\boldsymbol{\alpha}}, x}, Z_{T, \textbf{I}}^{{\boldsymbol{\alpha}}, x}\right)\Big| X^x_t = e\Big].
\end{aligned}
\end{equation*}
\normalsize
At the equilibrium, the HJB system will be coupled with a KFP system that represents the state dynamics and is written as follows by plugging in the equilibrium control
\small
\begin{equation*}
    \dot{p}^x(t, e) = \sum_{e^{\prime} \in E} q_t(e^{\prime}, e, (z)_{\textbf{K},\textbf{I}}, \hat{\phi}^x(t, e^{\prime})) p^x(t,e^{\prime}).
\end{equation*}
\normalsize

\end{proof}

We emphasize that the FBODE system consists of a \textit{continuum} of \textit{coupled} FBODEs. Our next step is to show the existence of the solution for the FBODE system which in turn gives a graphon Nash equilibrium.
\begin{theorem}[Existence of Nash]
    \label{existencetheorem}
    Under Assumption~\ref{assu:fbode_nash}
    \small
    \begin{equation*}
        \text{\normalsize if    } T\bar{\beta}\left(\max \left(\frac{1}{2}\left(\bar{\lambda}_{\mathbf{I}}+\bar{A}\right)^2, \bar{\lambda}_{\mathbf{K}}\right)+\frac{1}{2}(1+\bar{A})^2\right) < 1,
    \end{equation*}
    \normalsize
    there exists a bounded solution $(\bu,\bp)$ to the continuum of forward backward KFP-HJB system given in Theorem~\ref{the:fbode_nash}.
\end{theorem}

\begin{proof}
 The proof idea is similar to that of the existence theorem given in~\cite[Theorem 2]{graphon_epidemics}, differed by i) current setting requires an extension due to the model form with two aggregate variables, ii) a specific model of interest is analyzed, enabling a more refined conclusion. We give a sketch of proof in the main text.\footnote{The full proof can be found in Appendix~\ref{app:exproofdetail}.} We first show there exist unique processes $(Z_{t, \textbf{K}}^{x,(u,p)},Z_{t, \textbf{I}}^{x,(u,p)})_{t\in[0,T], x\in I}$ given fixed processes $(\bu, \bp) = (\bu^x, \bp^x)_{x\in I}$ by using Banach fixed point theorem. Second, we show that the mapping $(\bu, \bp) \mapsto (\hat{\bu}, \hat{\bp})$ where the mapping is defined by using the FBODE system in Theorem~\ref{the:fbode_nash} and the corresponding aggregates $(Z_{t, \textbf{K}}^{x,(u,p)},Z_{t, \textbf{I}}^{x,(u,p)})_{t\in[0,T], x\in I}$ (which is known to exist uniquely due to the first step) has a fixed point by using Schauder fixed point theorem. 
 The Banach fixed point theorem step will be extended as follows:

With solution space for the system properly defined, fix a pair $(u,p)$ in it, the aggregate mapping follows
\small
    \begin{equation*}
        \begin{aligned} & \Phi^{(u, p)}\left(\left(Z_{t, \textbf{K}}^x, Z_{t, \textbf{I}}^x\right)_{t \in [0, T], x \in I}\right)= \\ & \left( \int_I w(x, y) \hat{\phi}^y(t, \textbf{I}) p^y(t, \textbf{I}) d y, \int_I w(x, y) \hat{\phi}^y(t, \textbf{K}) p^y(t, \textbf{K}) d y\right)\end{aligned}
    \end{equation*}
    \normalsize
    We then prove the existence and uniqueness of fixed point for this functional when $\left(Z_{t, \textbf{K}}^x, Z_{t, \textbf{I}}^x\right)$ depends continuously on $(u,p)$ under some properly chosen aggregate space. Define the aggregate space 
    \small
    \begin{align*}
        \mathcal{Z}:=\{f \in C([0, T] ; &L^2(I ; \mathbb{R} \otimes \mathbb{R})): |f_t^x|_{[1]} \leq \bar{A}, |f_t^x|_{[2]} \leq \bar{A}, \\ &
        t \in[0, T], a.e. x \in I\}.
    \end{align*}
    \normalsize
    Given the 2-dimensional functional case, let 
    \small
    $$\|f\|_{\mathcal{Z}}:=\sup _{t \in[0, T]} \int_I|f(t)(x)| dx \quad \text{where $|\cdot|$ is 1-norm}.$$ 
    \normalsize
    One could easily check that $\|\cdot\|_{\mathcal{Z}}$ is a well-defined norm and $(\mathcal{Z}, \|\cdot\|_{\mathcal{Z}})$ is a closed subset of a Banach space; therefore, it is complete.
    Then, we prove $\Phi$ is a contraction mapping under the assumptions of Theorem~\ref{existencetheorem}, then we can apply Banach fixed point theorem. The remaining part that uses Schauder fixed point theorem is similar to ~\cite{graphon_epidemics}.
\end{proof}

\section{Example and Numerical Results} 
\label{sec:algo-num}

\noindent{\textbf{Example: Piecewise Constant Graphon. }} 
To compute the Nash equilibrium of the system, as an example we consider a setting with piecewise constant graphon. For nonnegative $m^{1}, \ldots, m^{K}$, player index $[0,1]$ was divided into $K$ intervals. We assume the players in the same group to be indistinguishable with the same model parameters and interaction strengths i.e., for $x, \tilde{x}$ from the same group, $\forall y \in I, w(x, y) = w(\tilde{x}, y)$. This regime enables the FBODE system to be finite dimensional. With discretized time, the algorithm iteratively updates the aggregates, value function, and distribution flows, $u$ and $p$, respectively, see Alg.~\ref{algo:pccgraphon}.

{
\begin{algorithm}
\caption{\small Piecewise Constant Graphon Game\label{algo:pccgraphon}}

\footnotesize
\textbf{Input:} Initial flows $(u,p)^{(0)}$, terminal time $T$, graphon $w$, other coefficients of the model.

\textbf{Output:} converged $(u^{*}, p^{*})$, equilibrium aggregates $(Z_{\textbf{K}}, Z_{\textbf{I}})$, equilibrium communication rates $\phi(t, \cdot)$.

\vskip1mm

\begin{algorithmic}[1]
\WHILE{{$\|u^{(k)}-u^{(k-1)}\| > \epsilon$} or {$\|p^{(k)}-p^{(k-1)}\| > \epsilon$} (at step k)}
    \STATE Compute aggregate $(Z_{\textbf{K}}, Z_\textbf{I})^{(k)}$ based on graphon
    \STATE Compute optimal control $\hat{\phi}^x(t, e)^{(k)}$
    \STATE Solve Kolmogorov-Fokker–Planck equation to achieve $p^{(k+1)}$
    \STATE Solve Hamilton–Jacobi–Bellman equation to achieve $u^{(k+1)}$
    \STATE Save variables in step 2-5 for the next iteration
\ENDWHILE

\RETURN $(u^{*}, p^{*}), (Z_{\textbf{K}}, Z_{\textbf{I}}), \phi(t, \cdot)$
\end{algorithmic}
\end{algorithm}}

\noindent{\textbf{Numerical Results.}}
We then provide some numerical results for an example inspired with real world data and parameters. We first study the effects of different policies on the rumor propagation in social media under the age groups. We divided population into 4 groups by age and set different base meeting intensity and forgetting rates for each group. Piecewise constant graphon weights and model parameters can be seen in tables~\ref{tab:e1graphon} and~\ref{tab:e1param}.
\small
\begin{table}[ht]
\centering
\begin{tabular}{|c|c|c|c|c|}
\hline
Age & 18-29 & 30-49 & 50-64 & 65+ \\ \hline
18-29 & 1 & 0.9 & 0.8 & 0.7 \\ \hline
30-49 & 0.9 & 0.9 & 0.8 & 0.8 \\ \hline
50-64 & 0.8 & 0.8 & 0.9 & 0.8 \\ \hline
65+ & 0.7 & 0.8 & 0.8 & 0.8 \\ \hline
\end{tabular}
\caption{\small Experiment 1 Graphon}
\label{tab:e1graphon}
\end{table}
\normalsize
\small
\begin{table}[ht]
\centering
\begin{tabular}{|c|c|c|c|c|c|}
\hline
Age & $\beta_{\textbf{S}}$ & $\beta_{\textbf{K}}$ & $\beta_{\textbf{I}}$ & $\mu_{\textbf{K}}$ & $\mu_{\textbf{I}}$ \\ \hline
18-29 & 0.4 & 0.5 & 0.75 & 0.1 & 0.1 \\ \hline
30-49 & 0.3 & 0.42 & 0.62 & 0.05 & 0.05 \\ \hline
50-64 & 0.3 & 0.32 & 0.48 & 0.05 & 0.05 \\ \hline
65+ & 0.3 & 0.2 & 0.3 & 0.15 & 0.15 \\ \hline
\end{tabular}
\caption{\small Experiment 1 Coefficients}
\label{tab:e1param}
\end{table}
\normalsize
The graphon is chosen based on the age social activity levels determined in \cite{graphon_epidemics}. Furthermore, it is natural to assume that the base meeting intensity is higher among younger individuals and the forgetting rate tends to increase with age. The meeting intensity for the groups at state $\textbf{K}$ and $\textbf{I}$ is determined based on the usage of popular social media platforms by people in different age groups \cite{Gottfried_2024}, where age group 18-29 $\beta_{\textbf{K}}$ is set benchmark as 0.5, and generalize to $\beta_{\textbf{I}}$'s via a sample ratio provided in \cite{rumorcas}. We set the initial state distribution flow to be $p_{0}(\textbf{S}) = 0.95$, $p_{0}(\textbf{K}) = 0.02$, $p_{0}(\textbf{I}) = 0.03$ and $p_{0}(\textbf{R}) = 0$. The terminal cost is set $c_j = 0$. Within this experiment, we conduct pairwise comparison between three policies. Recall that in the cost function, $\lambda_{\textbf{I}}$ and $\lambda_{\textbf{K}}$ are the two variables that could be adjusted by some regulators, representing \textit{penalty for spreading rumors} and \textit{reward for truth}, respectively. Our focus is on determining whether it is more effective to \textit{reward} or \textit{penalize} individuals based on their states in order to achieve more efficient rumor control. We therefore choose three different policy (cost function) parameters, \begin{itemize}
    \item \textbf{Policy 0}: $\lambda_{\textbf{K}} = 1.0$, $\lambda_{\textbf{I}} = 0.25$ (light on both)
    \item \textbf{Policy 1}: $\lambda_{\textbf{K}} = 1.5$, $\lambda_{\textbf{I}} = 0.25$ (stress on reward)
    \item \textbf{Policy 2}: $\lambda_{\textbf{K}} = 1.0$, $\lambda_{\textbf{I}} = 0.20$ (stress on penalty)
\end{itemize}
Note we did not use group-specific $\lambda$'s, which can be considered for further exploration. Policy 0 as benchmark is compared with the other two in Figures~\ref{fig:e1f1} and \ref{fig:e1f2}. We note Policy 2 significantly reduces the probability density of the Fake (\textbf{I}) state. In addition, a stronger reward policy (i.e. Policy 1) for truth effectively motivates middle-aged individuals to increase their communication rate.

\begin{figure}[t]
    \centering
    \includegraphics[width=1.0\linewidth]{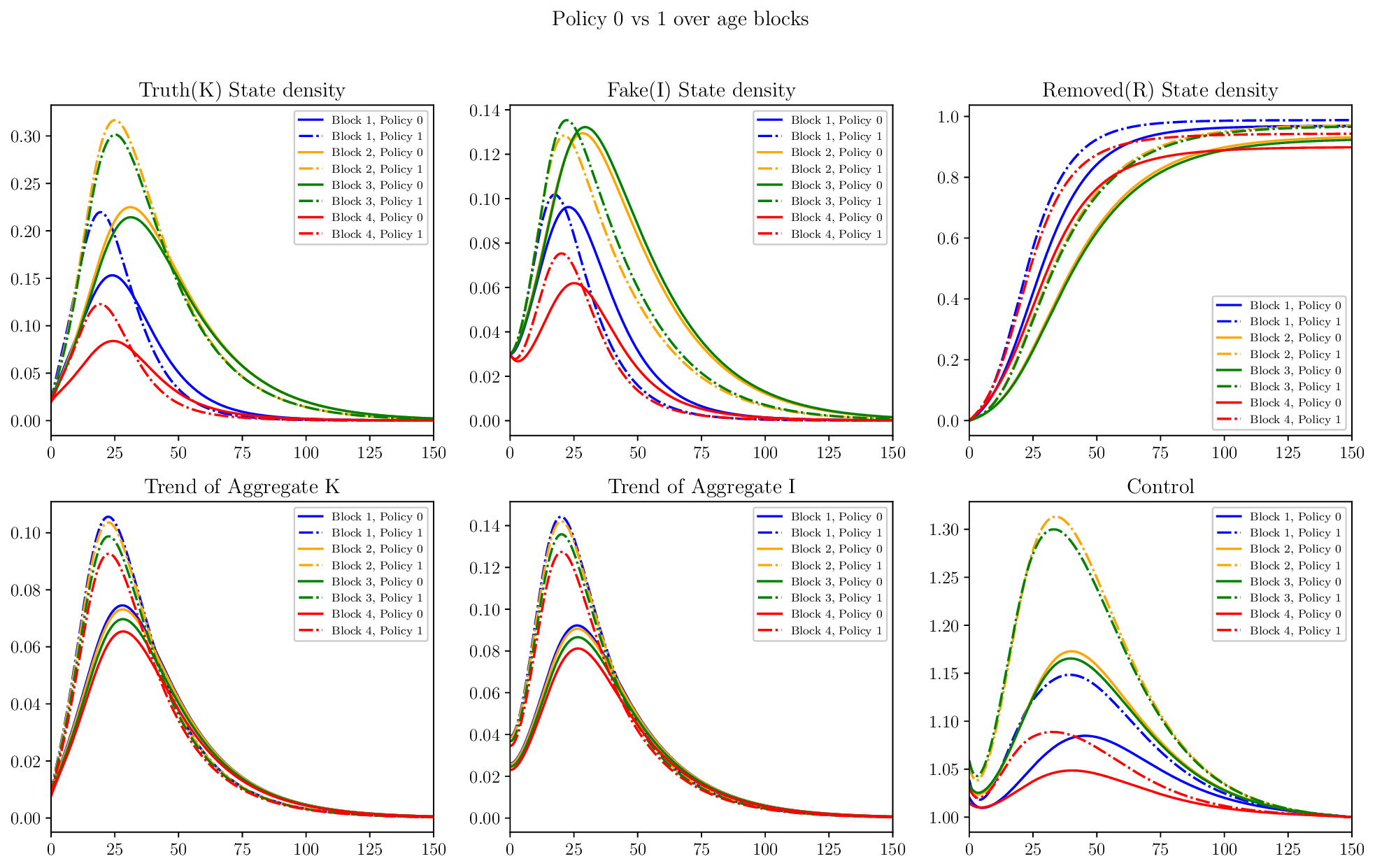}
    \caption{\small Policy 0 vs. Policy 1 w.r.t ages}
    \label{fig:e1f1}
\end{figure}

\begin{figure}[t]
    \centering
    \includegraphics[width=1.0\linewidth]{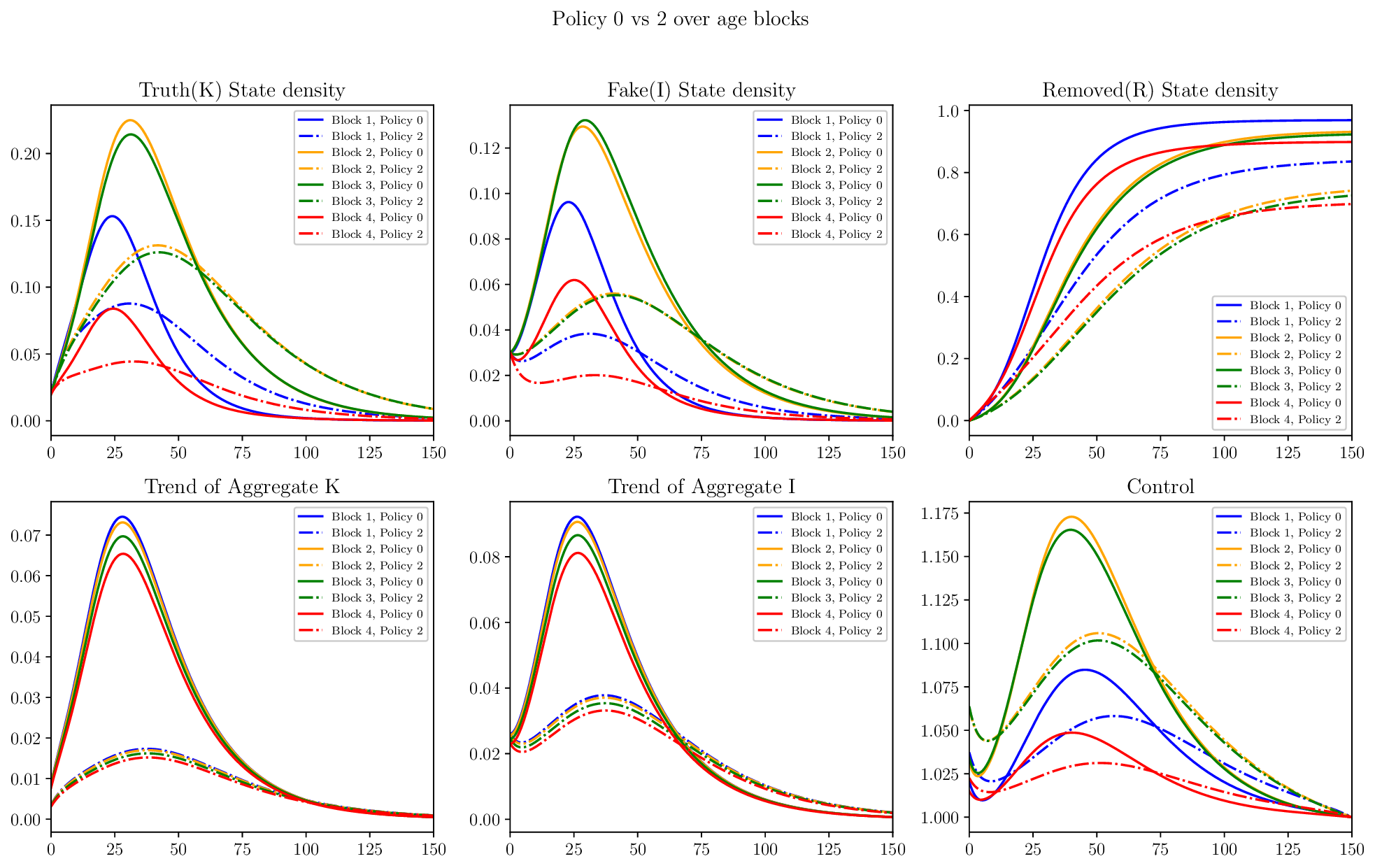}
    \caption{\small Policy 0 vs. Policy 2 w.r.t ages}
    \label{fig:e1f2}
\end{figure}


In the second experiment, we look for groups of popular social platforms and simulate with parameters retrieved from real world data. One may observe that this setting does not constitute strict \textit{grouping} among individuals, as individuals can use multiple social platforms simultaneously. 
To overcome this, we will assign each individual to the primary platform they use. 
\small
\begin{table}[ht]
\centering
\begin{tabular}{|c|c|c|c|c|}
\hline
Platform & Instagram & Facebook & Tiktok & Twitter \\ \hline
Instagram & 0.5 & 0.9 & 0.5 & 0.6 \\ \hline
Facebook & 0.9 & 1 & 0.85 & 0.55 \\ \hline
Tiktok & 0.5 & 0.85 & 0.5 & 0.4 \\ \hline
Twitter & 0.6 & 0.55 & 0.4 & 0.3 \\ \hline
\end{tabular}
\caption{\small Experiment 2 Graphon}
\label{tab:e2graphon}
\end{table}
\normalsize
\small
\begin{table}[ht]
\centering
\begin{tabular}{|c|c|c|c|c|c|}
\hline
\small Platform & $\beta_{\textbf{S}}$ & $\beta_{\textbf{K}}$ & $\beta_{\textbf{I}}$ & $\mu_{\textbf{K}}$ & $\mu_{\textbf{I}}$ \\ \hline
\small Instagram & 0.4 & 0.35 & 0.5 & 0.1 & 0.1 \\ \hline
\small Facebook & 0.4 & 0.5 & 0.75 & 0.1 & 0.1 \\ \hline
\small Tiktok & 0.4 & 0.25 & 0.35 & 0.1 & 0.1 \\ \hline
\small Twitter & 0.4 & 0.2 & 0.3 & 0.1 & 0.1 \\ \hline
\end{tabular}
\caption{\small Experiment 2 Coefficients}
\label{tab:e2coef}
\end{table}
\normalsize
The piecewise constant graphon weights (table~\ref{tab:e2graphon}) are established based on the reported proportion of individuals using two platforms simultaneously \cite{Sue_2021}, while the internal connectivity is decided with the reported social media ranking of news consumption \cite{West_2024}. Coefficients in table \ref{tab:e2coef} are chosen similarly to the first experiment. Contrary to the regulator, this experiment examines how two competing parties spread rumors, aiming to identify the fastest dissemination strategy. In other words, initial distribution is represented as a scheme here. We simulate two approaches: either dispersing the promotion cost across four platforms simultaneously (\textbf{Scheme 0}) or concentrating all resources on one platform (\textbf{Scheme 1,2,3,4} for four platforms). In addition, we add forgetting rate $\gamma := 0.1$ that goes from \textbf{R} back to \textbf{S} for further variability. The results can be seen in Figures~\ref{fig:e2f1} and \ref{fig:e2f2}. We note group 2 (facebook) that is more closely connected to other groups tend to remain more active in both the \textbf{K} and \textbf{I} states. Also, a concentrated outbreak of fake news on a single platform does not destabilize or collapse the system.

\begin{figure}[t]
    \centering
    \includegraphics[width=1.0\linewidth]{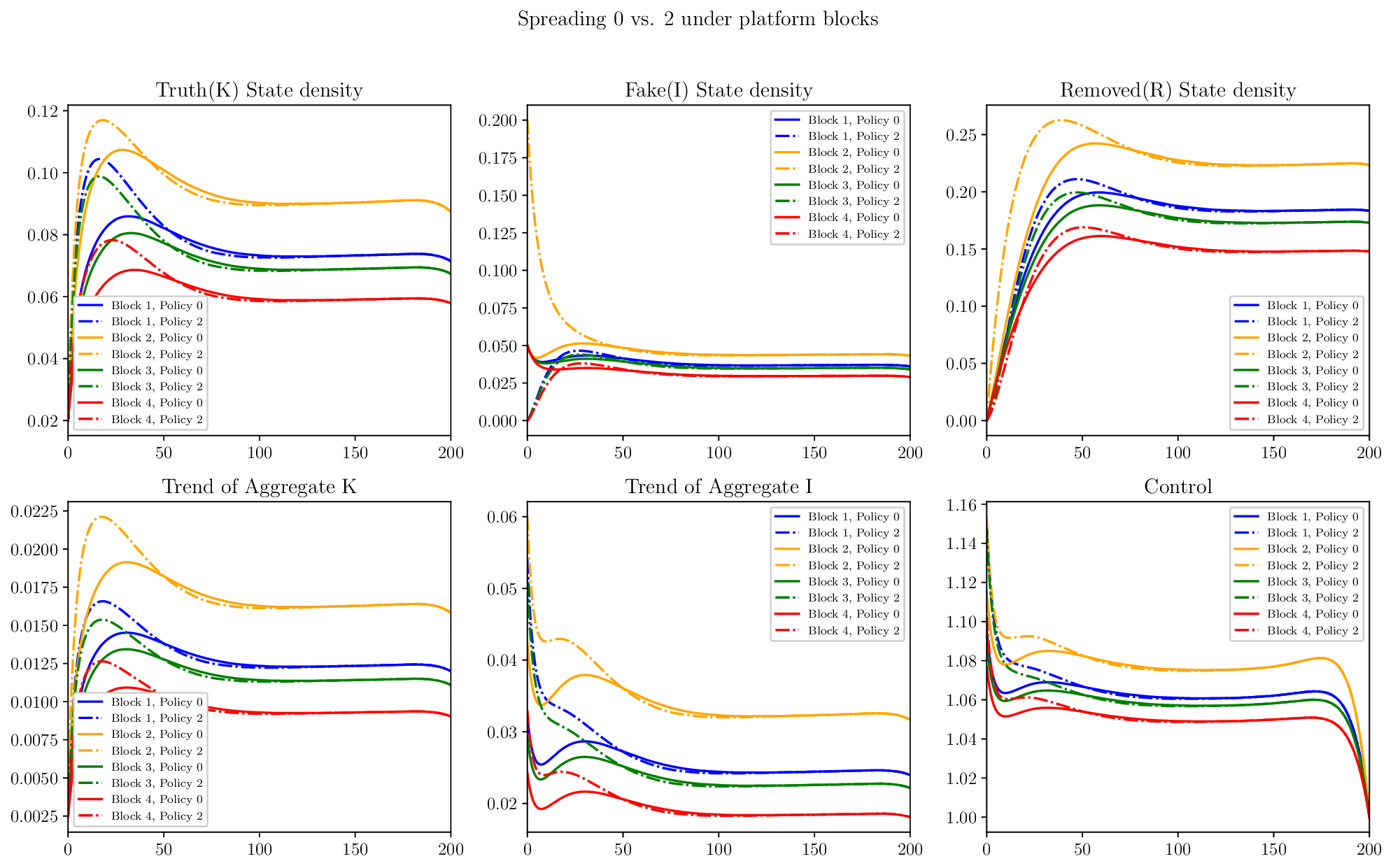}
    \caption{\small Scheme 0 vs. Scheme 2}
    \label{fig:e2f1}
    \vskip-4mm
\end{figure}

\begin{figure}[t]
    \centering
    \includegraphics[width=1.0\linewidth]{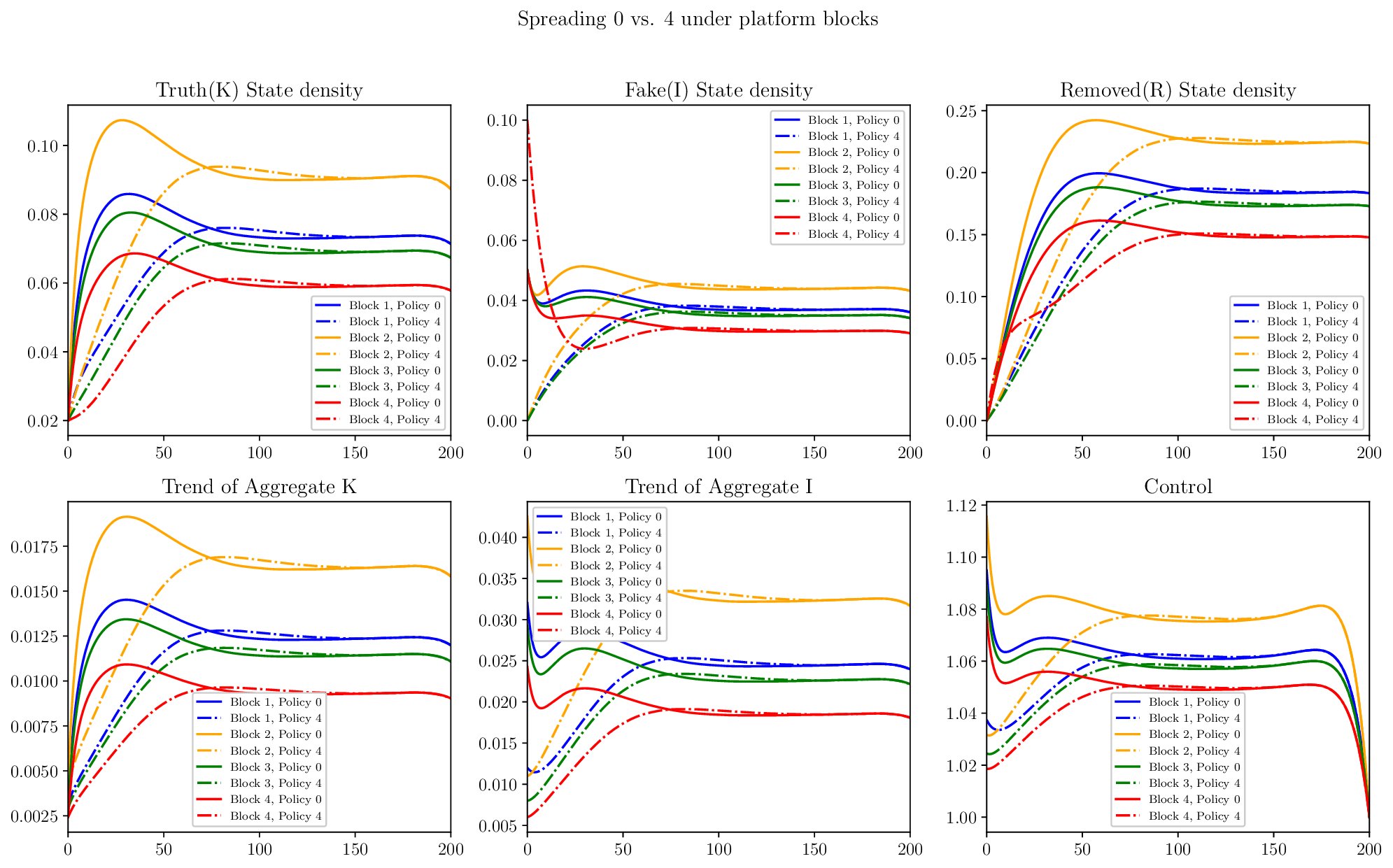}
    \caption{\small Scheme 0 vs. Scheme 4}
    \label{fig:e2f2}
    \vskip-4mm
\end{figure}

\section{Conclusions and Future Work}
\label{sec:conclusion-future}

{\bf Conclusions. } In this paper, we model how the policies affect the propagation of rumors (such as fake news) in a large population of non-cooperative agents that are interacting on a network (such as social network). Our rumor propagation model is inspired by the models from the system dynamics literature, where we extend the SKIR model by adding individual controls and weighted interactions with other individuals in the dynamics. We introduce the finite player model and the limiting graphon game model to approximate the Nash equilibrium when the number of players goes to infinity. We analyze the graphon game and characterize the graphon game Nash equilibrium with a \textit{continuum} of FBODE system. Furthermore, we prove the existence of the equilibrium. Finally, we introduce examples that use a piecewise constant graphon and analyze the numerical results to understand the effects of different policies.

{\bf Future Work. } Our future work has three main directions. First, we plan to implement numerical approaches to treat general graphons. Since using general graphons means that we need to solve the full continuum of FBODEs, other approaches need to be used such as deep learning~\cite{graphon_epidemics}. Second, we plan to use real life data to estimate the underlying network structure for the application of interest. Finally, we plan to implement a regulator who has their own objectives and analyze the Stackelberg equilibrium between the regulator and the graphon game population.

\bibliographystyle{ieeetr}

\bibliography{References.bib}

\begin{thebibliography}{10}

\bibitem{daley-kendall}
D.~J. DALEY and D.~G. KENDALL, ``{Stochastic Rumours},'' {\em IMA Journal of Applied Mathematics}, vol.~1, pp.~42--55, 03 1965.

\bibitem{Maki1973MathematicalMA}
D.~P. Maki and M.~Thompson, ``Mathematical models and applications : with emphasis on the social, life, and management sciences,'' 1973.

\bibitem{6785909}
Y.~Bao, C.~Yi, Y.~Xue, and Y.~Dong, ``A new rumor propagation model and control strategy on social networks,'' in {\em 2013 IEEE/ACM International Conference on Advances in Social Networks Analysis and Mining (ASONAM 2013)}, pp.~1472--1473, 2013.

\bibitem{Bettencourt_2006}
L.~M. Bettencourt, A.~Cintrón-Arias, D.~I. Kaiser, and C.~Castillo-Chávez, ``The power of a good idea: Quantitative modeling of the spread of ideas from epidemiological models,'' {\em Physica A: Statistical Mechanics and its Applications}, vol.~364, p.~513–536, May 2006.

\bibitem{XIA2015295}
L.-L. Xia, G.-P. Jiang, B.~Song, and Y.-R. Song, ``Rumor spreading model considering hesitating mechanism in complex social networks,'' {\em Physica A: Statistical Mechanics and its Applications}, vol.~437, pp.~295--303, 2015.

\bibitem{optc}
M.~Askarizade, S.~Najafi, and B.~Tork~Ladani, ``Modeling and analysis of rumor control strategies in social networks,'' {\em Computer and Knowledge Engineering}, vol.~5, no.~1, pp.~59--68, 2022.

\bibitem{DONG2022112711}
Y.~Dong, L.~Huo, and L.~Zhao, ``An improved two-layer model for rumor propagation considering time delay and event-triggered impulsive control strategy,'' {\em Chaos, Solitons \& Fractals}, vol.~164, p.~112711, 2022.

\bibitem{dif_psych}
N.~Difonzo and P.~Bordia, ``Rumor psychology: Social and organizational approaches,'' 01 2007.

\bibitem{MFGG}
J.-M. Lasry and P.-L. Lions, ``Mean field games,'' {\em Japanese Journal of Mathematics}, vol.~2, pp.~229--260, 03 2007.

\bibitem{huang2006large}
M.~Huang, R.~P. Malham{\'e}, P.~E. Caines, {\em et~al.}, ``{Large population stochastic dynamic games: closed-loop McKean-Vlasov systems and the Nash certainty equivalence principle},'' {\em Communications in Information \& Systems}, vol.~6, no.~3, pp.~221--252, 2006.

\bibitem{MFInfection}
J.~Arabneydi and A.~G. Aghdam, ``A mean-field team approach to minimize the spread of infection in a network,'' in {\em 2019 American Control Conference (ACC)}, pp.~2747--2752, 2019.

\bibitem{DBLP:journals/corr/abs-1802-00080}
F.~Parise and A.~Ozdaglar, ``Graphon games: A statistical framework for network games and interventions,'' {\em Econometrica}, vol.~91, no.~1, pp.~191--225, 2023.

\bibitem{carmona2019stochasticgraphongamesi}
R.~Carmona, D.~B. Cooney, C.~V. Graves, and M.~Lauri\`{e}re, ``Stochastic graphon games: I. the static case,'' {\em Mathematics of Operations Research}, vol.~47, no.~1, pp.~750--778, 2022.

\bibitem{caines_graphon}
P.~E. Caines and M.~Huang, ``Graphon mean field games and their equations,'' {\em SIAM Journal on Control and Optimization}, vol.~59, no.~6, pp.~4373--4399, 2021.

\bibitem{carmona_lq_graphon}
A.~Aurell, R.~Carmona, and M.~Lauri{\`e}re, ``Stochastic graphon games: Ii. the linear-quadratic case,'' {\em Applied Mathematics \& Optimization}, vol.~85, no.~3, p.~39, 2022.

\bibitem{bayraktar_graphon}
E.~Bayraktar, S.~Chakraborty, and R.~Wu, ``{Graphon mean field systems},'' {\em The Annals of Applied Probability}, vol.~33, no.~5, pp.~3587 -- 3619, 2023.

\bibitem{graphon_epidemics}
A.~Aurell, R.~Carmona, G.~Dayanıklı, and M.~Lauri{\`e}re, ``Finite state graphon games with applications to epidemics,'' {\em Dynamic Games and Applications}, vol.~12, no.~1, pp.~49--81, 2022.

\bibitem{xiao}
Y.~Xiao, D.~Chen, S.~Wei, Q.~Li, H.~Wang, and M.~Xu, ``Rumor propagation dynamic model based on evolutionary game and anti-rumor,'' {\em Nonlinear Dynamics}, vol.~95, 01 2019.

\bibitem{tani}
T.~Tani, ``Legal {Responsibility} for {Fake} {News},'' {\em Journal of International Media \& Entertainment Law}, vol.~8, no.~2, 2020.

\bibitem{chinesenews}
A.~Kharpal, ``Chinese police arrest man who allegedly used chatgpt to spread fake news in first case of its kind,'' May 9, 2023.

\bibitem{uslaw}
``18 u.s. code § 35 - imparting or conveying false information.'' \url{https://www.law.cornell.edu/uscode/text/18/35#:~:text=Whoever%20imparts%20or%20conveys%20or,shall%20be%20subject%20to%20a}, 1994.

\bibitem{Alonzo2021-yb}
D.~Alonzo and M.~Popescu, ``Utilizing social media platforms to promote mental health awareness and help seeking in underserved communities during the {COVID-19} pandemic,'' {\em J Educ Health Promot}, vol.~10, p.~156, May 2021.

\bibitem{carmona2018probabilistic}
R.~Carmona and F.~Delarue, {\em Probabilistic Theory of Mean Field Games with Applications I: Mean Field FBSDEs, Control, and Games}.
\newblock Probability Theory and Stochastic Modelling, Springer International Publishing, 2018.

\bibitem{Gottfried_2024}
J.~Gottfried, ``Americans’ social media use.'' \url{https://www.pewresearch.org/internet/2024/01/31/americans-social-media-use/}, Jan 2024.

\bibitem{rumorcas}
A.~Friggeri, L.~Adamic, D.~Eckles, and J.~Cheng, ``Rumor cascades,'' {\em Proceedings of the 8th International Conference on Weblogs and Social Media, ICWSM 2014}, vol.~8, pp.~101--110, 05 2014.

\bibitem{Sue_2021}
{Sue}, ``Social {Media} {Overlap}.'' \url{https://fusionmediaagency.com/blog/social-media-overlap/}, June 2021.

\bibitem{West_2024}
C.~West, ``24 facebook statistics marketers should know in 2024.'' \url{https://sproutsocial.com/insights/facebook-stats-for-marketers/#usage}, Mar 2024.

\bibitem{brezis2010functional}
H.~Brezis, {\em Functional Analysis, Sobolev Spaces and Partial Differential Equations}.
\newblock Universitext, Springer New York, 2010.

\end{thebibliography}

\newpage
\onecolumn

\appendices 

\section{Detailed proof of Theorem~\ref{existencetheorem}}
\label{app:exproofdetail}

We start by defining the solution space for $(u,p)$. Denote $\mathcal{K}_{C_1}$ to be a closed ball of all continuous functions that maps $[0,T]$ to $L^{2}(I \times E)\times L^2(I \times E)$. Some natural constraints includes $p^x(t, e) \geqslant 0$ and $\sum_{e \in E} p^x(t, e)=1$ for all $t \in [0,T]$, $x \in I$ and $e \in E$. In addition let $(u,p)$ bounded by some $C_1 > 0$ with uniform norm. That is
    $$\mathcal{K}_{C_{1}}:=\left\{(u, p) \in C\left([0, T] ; L^2(I \times E) \times L^2(I \times E)\right): \sum_{e \in E} p^x(t, e)=1, p^x(t, e) \geqslant 0 \ \forall e \in E, \|(u, p)\|_{\infty} \leqslant C_1 \right\}$$
By assumption on $f$ and $Q$-matrix, the mapping $\alpha \mapsto H^x\left(t, e,\left(z_{\bK}, z_{\bI}\right), h, \alpha\right)$ is strongly convex. Its unique minimizer is the optimal control $\hat{\phi}^x(t, e)$, $x \in I$ and $e \in E$. \\ Fix $(u, p) \in \mathcal{K}_{C_{1}}$, define the aggregate mapping
    \begin{equation}
        \begin{aligned} & \Phi^{(u, p)}\left(\left(Z_{t, \textbf{K}}^x, Z_{t, \textbf{I}}^x\right)_{t \in [0, T], x \in I}\right)= \left( \int_I w(x, y) \hat{\phi}^y(t, \textbf{I}) p^y(t, \textbf{I}) d y, \int_I w(x, y) \hat{\phi}^y(t, \textbf{K}) p^y(t, \textbf{K}) d y\right)_{t \in [0, T], x \in I}\end{aligned}
    \end{equation}
    Define the aggregate space, where $|\cdot|$ denotes standard 1-norm, 
    \begin{equation}
        \mathcal{Z}:=\left\{f \in C\left([0, T] ; L^2(I ; \mathbb{R} \otimes \mathbb{R})\right):\left|f_t^x\right|_{[1]} \leq \bar{A} \text{ and } \left|f_t^x\right|_{[2]} \leq \bar{A}, t \in[0, T]\right., a.e. \left.x \in I\right\}
    \end{equation}
    Given the 2-dimensional functional case, let $$\|f\|_{\mathcal{Z}}:=\sup _{t \in[0, T]} \int_I|f(t)(x)| dx$$ 
    One could check that $\|\cdot\|_{\mathcal{Z}}$ is a well-defined norm. It follows $\Phi^{(u,p)}(\mathcal{Z}) \subset \mathcal{Z}$ by boundedness of control. Then since $C\left([0, T] ; L(I ; \mathbb{R} \otimes \mathbb{R})\right)$ is complete, $\mathcal{Z}$ as its closed subset is complete. The completeness of $\mathcal{Z}$ is briefly stated and proved in lemma \ref{closednesslemma}. Then we show $\Phi^{(u, p)}$ is a contraction mapping under this construction. \\
    We first state the local Lipschitz continuity (with computed Lipschitz constant) of the optimal Markovian control. Note by assumption, $|f^{x}(t, e, (z_{\bK}, z_{\bI}), a)|$ is bounded by $L_{f}(\bar{A}) := \max{\left(\frac{1}{2}(\bddlambdaI + \bar{A})^2, \bddlambdaK\right)} + \frac{1}{2}(1+\bar{A})^2$, and $|g^{x}(e, (z_{\bK}, z_{\bI}), a)|$ is bounded by $\bar{c}$ for any $x \in I$, $e \in E$. By definition of value function, \begin{equation}\label{ubound}
        |u^x(t,\cdot)| \leq (T-t)\left(\max{\left(\frac{1}{2}(\bddlambdaI + \bar{A})^2, \bddlambdaK\right)} + \frac{1}{2}(1+\bar{A})^2\right) + \bar{c} = (T-t)L_{f}(\bar{A}) + \bar{c}
    \end{equation}\\
    Consider the lipschitz continuity of optimal control, for $\left(t, e, z_{\mathbf{K}}, z_{\mathbf{I}}\right)$ and $\left(t', e, z'_{\mathbf{K}}, z'_{\mathbf{I}}\right)$ in $[0,T] \times E \times [0,\bar{A}] \times [0,\bar{A}]$,
    \begin{equation}
        \label{lipsahat}
        \begin{aligned}
            |\hat{a}_{e}^x\left(t,\left(z_{\bK}, z_{\bI}\right), u^{x}(t,\cdot)\right) - \hat{a}_{e}^x\left(t',\left(z'_{\bK}, z'_{\bI}\right), u^{x}(t',\cdot)\right)| & \leq L_{\lambda_{\bI}}|t - t'| + T\bar{\beta}L_{f}(\bar{A})\left(|z_{\bK}-z_{\bK}'|+|z_{\bI}-z_{\bI}'|\right) \\ & =: L_{\lambda_{\bI}}|t - t'| + L_{\hat{a}}(\bar{A},T)\left(|z_{\bK}-z_{\bK}'|+|z_{\bI}-z_{\bI}'|\right)
        \end{aligned}
    \end{equation}
    With fixed $(u, p)$, consider $\left(Z_{t, \textbf{K}}^{x,1}, Z_{t, \textbf{I}}^{x,1}\right)_{t \in [0, T], x \in I}$ and $\left(Z_{t, \textbf{K}}^{x,2}, Z_{t, \textbf{I}}^{x,2}\right)_{t \in [0, T], x \in I}$ in $\mathcal{Z}$. Their distance follows
    \begin{equation}
    \begin{aligned}
        &\left\| \Phi^{(u, p)}\left(\left(Z_{t, \textbf{K}}^{x,1}, Z_{t, \textbf{I}}^{x,1}\right)_{t \in [0, T], x \in I}\right) - \Phi^{(u, p)}\left(\left(Z_{t, \textbf{K}}^{x,2}, Z_{t, \textbf{I}}^{x,2}\right)_{t \in [0, T], x \in I}\right) \right\|_{\mathcal{Z}} \\
        &
        \leqslant \sup_{t \in [0, T]} \int_I \int_I w(x, y)\left|\hat{\phi}^{y, 1}(t, \textbf{I}) p^{y, 1}(t, \textbf{I})-\hat{\phi}^{y, 2}(t, \textbf{I}) p^{y,2}(t, \textbf{I})\right| dy + \\& \quad \quad \quad \quad \quad \quad \int_I w(x, y)\left|\hat{\phi}^{y, 1}(t, \textbf{K}) p^{y, 1}(t, \textbf{K})-\hat{\phi}^{y, 2}(t, \textbf{K}) p^{y,2}(t, \textbf{K})\right| dy \ dx
        \\
        &\leqslant \|\omega\|_{L^2(I \times I)} L_{\hat{a}}(\bar{A},T) \sup _{t \in[0, T]}\int_I\left|Z_{t, \textbf{K}}^{x, 1}-Z_{t, \textbf{K}}^{x, 2}\right|+\left|Z_{t, \textbf{I}}^{x, 1}-Z_{t, \textbf{I}}^{x, 2}\right| d x \\
        &\leqslant L_{\hat{a}}(\bar{A},T) \|Z_{t, \textbf{K}}^{x, 1}-Z_{t, \textbf{K}}^{x, 2}, Z_{t, \textbf{I}}^{x, 1}-Z_{t, \textbf{I}}^{x, 2} \|_{\mathcal{Z}}
    \end{aligned}
    \end{equation}

    Thus apply Banach fixed point theorem, we have shown the existence and uniqueness of fixed point for $\Phi$ on $\mathcal{Z}$, denoted $\hat{Z}^{(u, p)} :=\left(\hat{Z}_{t, \textbf{K}}^{(u, p), x}, \hat{Z}_{t, \textbf{I}}^{(u, p), x}\right)_{t \in [0,T], x \in I}$. \\
    Fixing $(\hat{Z}(u, p))_{\textbf{K}, \textbf{I}}$ and $u$, we solve the Kolmogorov equation to obtain the solution $\hat{p}$. The Cauchy-Lipschitz-Picard theorem guarantees both the existence and uniqueness of $\hat{p}$ (\cite{brezis2010functional}, Theorem 7.3) by considering $q$ as a linear operator on the Banach space $L(I \times E)$. And by boundedness of entries of Q-matrix by assumptions, the time derivative of $\hat{p}$ is bounded. Therefore, we can conclude that $\hat{p}$ is equicontinuous.\\
    Then fixing $(\hat{Z}(u, p))_{\bK, \bI}$ and $\hat{p}$, we solve the HJB equation to obtain the solution $\hat{u}$. Similarly, the Cauchy-Lipschitz-Picard theorem ensures both the existence and uniqueness of $\hat{u}$ by treating $\hat{H}$ as a Lipschitz operator on the Banach space $L(I \times E)$ by lipschitz continuity of $\lambda_{\bI}(t)$. Additionally, by boundedness of entries of Q-matrix, $\lambda_{\bI}(t)$ and $c^{x}$, there is a uniform bound on the time derivative of $\hat{u}$ because the Hamiltonian $\hat{H}$ is bounded under model setup. Therefore, we can conclude that $\hat{u}$ is equicontinuous. 
    
    For large enough $C_1$ the map $\Psi: \mathcal{K}_{C_1} \ni(u, p) \mapsto(\hat{u}, \hat{p})$ has the same domain and codomain. Then by Arzela-Ascoli theorem, $\Psi: \mathcal{K}_{C_1}$ is compact. To apply Schauder theorem for $\Psi$, we show the continuity of $\hat{Z}$, $\hat{p}$ and $\hat{u}$.

    We start by $(\hat{Z}_{\bK}, \hat{Z}_{\bI})$. Note as joint convergence implies marginal convergence under $\cL^{1}$, it suffices to show the joint convergence. Denote $(\hat{Z}^{n, x}_{\bK}, \hat{Z}^{n, x}_{\bI})$ as the aggregate induced by $(u^n, p^n)$, it follows
    \begin{equation}
        \begin{aligned}
            \sup _{t \in[0, T]} \int_I & \left|\hat{Z}_{t, \bI}^{n, x}-\hat{Z}_{t, \bI}^x\right|+\left|\hat{Z}_{t, \bK}^{n, x}-\hat{Z}_{t, \bK}^x\right| dx = \\
            & \sup _t \int_I \left(\left|\int_I w(x, y) \hat{\phi}^{n, y}(t, \bI) p^{n, y}(t, \bI)-w(x, y) \hat{\phi}^{y}(t,\bI) p^y(t, \bI) d y\right| \right. + \\ & \left. \left|\int_I w(x, y) \hat{\phi}^{n, y}(t, \bK) p^{n, y}(t, \bK)-w(x, y) \hat{\phi}^{y}(t,\bK) p^y(t, \bK) d y\right| \right) dx\\
            &\leq \sup_{t} \int_{I} \left|\hat{\phi}^{n, y}(t, \bI)-\hat{\phi}^y(t, \bI)\right| p^{n, y}(t, \bI) + \hat{\phi}^y(t, \bI) \left| p^{n, y}(t, \bI) - p^{y}(t, \bI) \right| + \\& \left|\hat{\phi}^{n, y}(t, \bK)-\hat{\phi}^y(t, \bK)\right| p^{n, y}(t, \bK) + \hat{\phi}^y(t, \bK) \left| p^{n, y}(t, \bK) - p^{y}(t, \bK) \right| dy
        \end{aligned}        
    \end{equation}
    The bound goes to 0 as $n \rightarrow \infty$ by the continuity of $\hat{\phi}^{y}$. Then we show the continuity of $\hat{p}$. For $s \in [0,T]$, it follows
    \begin{equation}
        \begin{aligned}
            \int_I\left|\hat{p}^{n, x}(s, .)-\hat{p}^x(s, \cdot)\right| d x & \leq
         \int_I \int_0^s \sum_{e, e^{\prime} \in E} \left| q_{e^{\prime}, e}^x\left(\hat{\phi}^{n, x}\left(t, e^{\prime}\right), \hat{Z}_{t, \bK}^{n, x}, \hat{Z}_{t, \bI}^{n, x}\right)\hat{p}^{n,x}(t,e') - q_{e^{\prime}, e}^x\left(\hat{\phi}^{x}\left(t, e^{\prime}\right), \hat{Z}_{t, \bK}^{x}, \hat{Z}_{t, \bI}^{x}\right)\hat{p}^{x}(t,e') \right| dt dx\\
         & \leq \int_I \int_0^s \sum_{e, e^{\prime} \in E}
         \left|q_{e^{\prime}, e}^x\left(\hat{\phi}^{n, x}\left(t, e^{\prime}\right), \hat{Z}_{t, \bK}^{n, x}, \hat{Z}_{t, \bI}^{n, x}\right)\right|\left|\hat{p}^{n,x}(t,e') - \hat{p}^{x}(t,e')\right| + \\
         & \left|q_{e^{\prime}, e}^x\left(\hat{\phi}^{n, x}\left(t, e^{\prime}\right), \hat{Z}_{t, \bK}^{n, x}, \hat{Z}_{t, \bI}^{n, x}\right) - q_{e^{\prime}, e}^x\left(\hat{\phi}^{x}\left(t, e^{\prime}\right), \hat{Z}_{t, \bK}^{x}, \hat{Z}_{t, \bI}^{x}\right)\right| \left|\hat{p}^{x}(t,e')\right| dt dx \\
         & \leq C\left( \int_0^s \sum_{e, e^{\prime} \in E} \int_I \left| \nqq - \qq \right| dx dt + \right. \\ & \left. \int_{0}^{s} \int_{I} \left| \npp - \pp \right| dxdt \right)
        \end{aligned}
    \end{equation}

    The last inequality holds by the boundedness of $\hat{p}$ and $q$, where $C$ is a constant of $A$ and $\bar{\beta}$. Then by Grönwall's inequality, we have $$\int_{I} \left| \npp - \pp \right| dx \leq C \int_0^s \sum_{e, e^{\prime} \in E} \int_I \left| \qq - \nqq \right| dx dt$$   
The bound goes to 0 as $n \rightarrow \infty$ by the the continuity of $\hat{\phi}^{x}$ and the model form.
We lastly prove the continuity of $\hat{u}$. Note
\begin{equation}
    \begin{aligned}
    \int_I\left|\hat{u}^{n, x}(s, \cdot)-\hat{u}^x(s, \cdot)\right| d x
& \leq \int_I \int_s^T \left|-H^x\left(t, e,\left(\hat{Z}_{t, \bK}^{n, x}, \hat{Z}_{t, \bI}^{n, x}\right), \Delta_e \hat{u}^{n, x}(t, \cdot)\right)
+H^x\left(t, e,\left(\hat{Z}_{t, \bK}^{n, x}, \hat{Z}_{t, \bI}^n\right), \Delta_e \hat{u}^x(t, \cdot)\right)\right| d t d x \\
& \leq \int_I \int_s^T \sum_{s,s' \in E} \left|-f^x\left(t, e,\left(\hat{Z}_{t, \bK}^{n, x}, \hat{Z}_{t, \bI}^{n,x}\right), \hat{\phi}^{n, x}\left(t, e^{\prime}\right)\right)+f^x\left(t, e,\left(\hat{Z}_{t, \bK}^x, \hat{Z}_{t, \bI}^x\right), \hat{\phi}^x\left(t, e^{\prime}\right)\right) \right| d t d x \\
& + \int_I \int_s^T \sum_{s,s' \in E} \left| \qq - \nqq \right| \left| \uu \right| dt dx \\
& + \int_I \int_s^T \sum_{s,s' \in E} \left| \qq \right| \left| \uu - \nuu \right| dt dx \\
& \leq \cdots + C \int_I \int_s^T \left| \uu - \nuu \right| dt dx
    \end{aligned}
\end{equation}
Similarly, the second inequality uses the boundedness of entries of $Q$-matrix; and $\cdots$ represents the first two terms in a fold. Thus by Grönwall's inequality, it follows
\begin{equation}
    \begin{aligned}
        \int_I\left|\hat{u}^{n, x}(s, \cdot)-\hat{u}^x(s, \cdot)\right| d x\leq & C\left(\int_I \int_s^T\left|q_{e, e^{\prime}}\left(\hat{\phi}^{x}(t,e), \hat{Z}_{t, \bK}^x, \hat{Z}_{t, \bI}^x\right)-q_{e, e^{\prime}}\left(\hat{\phi}^{n,x}(t,e), \hat{Z}_{t, \bK}^{n,x}, \hat{Z}_{t, \bI}^{n,x}\right)\right|^2 d t d x\right. \\ & \left.+\int_I \int_s^T\left|-f^x\left(t, e, \hat{Z}_{t, \bK}^{n,x}, \hat{Z}_{t, \bI}^{n,x}, \hat{\phi}^{n,x}(t,e)\right)+f^x\left(t, e, \hat{Z}_{t, \bK}^{x}, \hat{Z}_{t, \bI}^{x}, \hat{\phi}^{x}(t,e)\right)\right|^2 d t d x\right)
    \end{aligned}
\end{equation}
Again the continuity of $(\hat{Z}_{\bK}, \hat{Z}_{\bI})$ induces the convergence of the two $q$ functions. This completes the proof.

    \begin{lemma}
        \label{closednesslemma}
        $\mathcal{Z}$ is closed in $C\left([0, T] ; L^2(I ; \mathbb{R} \otimes \mathbb{R})\right)$ equipped with norm $\|\cdot\|_{\mathcal{Z}}$.
    \end{lemma}
    \begin{proof}
        First note $\mathcal{Z}$ is closed if and only if any convergent sequence $(f_n)_{n=1}^{\infty}$ in $\mathcal{Z}$ converges to some $f \in \mathcal{Z}$. For the sake of contradiction, suppose there exists $(f_n)_{n=1}^{\infty} \in \mathcal{Z}$ converges to $f \notin \mathcal{Z}$. Then there exists $t_{0} \in [0,T]$ such that there exists a measurable set $\tilde{I} \subset I$ with $m(\tilde{I}) > 0$, that $f(t_0)(x) > \bar{A} + \epsilon_0$ or $f(t_0)(x) < -\bar{A} -\epsilon_0$ for some $\epsilon_0 > 0$, $x \in \tilde{I}$. Here $m$ is a measure equipped on $I$. Then $\forall n \in \mathbb{N}$, $$\int_{I}\left|f_n\left(t_0\right)(x)-f\left(t_0\right)(x)\right| d x \geq \int_{\tilde{I}}\left|f_n\left(t_0\right)(x)-f\left(t_0\right)(x)\right| d x>\epsilon_0 \cdot m(\tilde{I})$$
        There is contradiction and thus completes the proof.
    \end{proof}

\end{document}